\documentclass[12pt]{article}
\usepackage[dvips]{graphics}
\usepackage{a4,makeidx}
\usepackage{amsfonts}
\usepackage{amsmath}
\usepackage{amssymb}
\usepackage{graphics}
\usepackage{graphicx}
\usepackage{ifthen}
\usepackage{inputenc}
\usepackage{latexsym}
\usepackage{makeidx}
\usepackage{syntonly}
\usepackage{amsthm}
\usepackage{lipsum}

\def\constr#1^#2{\mathrel{\mathop{\kern 0pt#1}\limits^{#2}}}
\def\build#1_#2{\mathrel{\mathop{\kern 0pt#1}\limits_{#2}}}
  
at 8pt

\newtheorem{theorem}{Theorem}[section]
\newtheorem{definition}{Definition}[section]

\newtheorem{proposition}{Proposition}[section]
\newtheorem{lemma}{Lemma}[section]
\newtheorem{remark}{Remark}[section]
\newtheorem{example}{Example}[section]
\numberwithin{equation}{section}

\def\fnt#1#2{\footnotetext{\kern-.3em
    {$^{\mbox{\scriptsize #1}}$}{#2}}}

\newcommand{\N}{\mathcal{N}}

\chardef\bslchar=`\\ 
\makeatletter
\newcommand{\addbslash}{\expandafter\@addbslash\string}
\def\@addbslash#1{\bslchar\@nobslash#1}
\newcommand{\nobslash}{\expandafter\@nobslash\string}
\def\@nobslash#1{\ifnum`#1=\bslchar\else#1\fi}
\newcommand{\ntt}{\normalfont\ttfamily}
\def\@boxorbreak{\leavevmode
  \ifmmode\hbox\else\ifdim\lastskip=\z@\penalty9999 \fi\fi}
\DeclareRobustCommand{\cs}[1]{\@boxorbreak{\ntt\addbslash#1\@empty}}
\makeatother 

\title{\bf GENERALIZATIONS OF KAPLANSKY THEOREM FOR SOME $(p,k)$-QUASI-HYPONORMAL OPERATORS  }
\author{  Abdelkader Benali $^{[1]}$ and Ould Ahmed Mahmoud Sid Ahmed
$^{[2]}$\\
\\
$^{[1]}$ Mathematics Department,Faculty of science,University of
Hassiba\\ Benbouali,
 Chlef Algeria. B.P. 151 Hay Essalem, chlef 02000, Algeria.
 \\benali4848@gmail.com\\
\noindent $^{[2]}$ Mathematics Department, College of Science.
Aljouf University\\Aljouf 2014. Saudi Arabia
 \\
 sidahmed@ju.edu.sa
}
\begin{document}

\maketitle\
\maketitle \
\begin{abstract}
In the present paper, we generalized some notions of bounded
operators to unbounded operators on Hilbert space such as
$k$-quasihyponormal and $k$-paranormal unbounded operators.
Furthermore, we extend the Kaplansky theorem for normal operators to
some $(p,k)$-quasihyponormal operators. Namely the
$(p,k)$-quasihyponormality of the product $AB$ and $BA$ for two
operators.
\end{abstract}

\maketitle {\bf Keywords.} Unbounded operator,normal operator,
$(k,p)$-quasihyponormal operator,$k$-paranormal operator.

{\bf Mathematics Subject Classification (2010).} Primary 47B15,
Secondary 46L10.

\begin{center}
\section{INTRODUCTION}
\end{center}
Through out the paper we denote Hilbert space over the field of
complex numbers $\mathbb{C}$ by $\mathcal{H}$ and the usual inner
product and the corresponding norm of $\mathcal{H}$ are denoted by
$\langle\; ,\; \rangle$ and $\|\;.\;\|$ respectively. Let us fix
some more notations.We write $\mathcal{B}(\mathcal{H})$
 for the set of all bounded linear operators in $\mathcal{H}$ whose domain are
equal to$\mathcal{H}$ .
 For an
operator $A \in \mathcal{B}(\mathcal{H}) $ ,the
 range, the kernel  and the adjoint of $A$ are
denoted by $\mathcal{R}(A)$,\;$\mathcal{N}(A)$ and $A^*$
respectively.
  If $\mathcal{M}$ is a
subspace of $\mathcal{H}$, $\overline{\mathcal{M}}$ and
$\mathcal{M}^\bot$ denote its closure and its orthogonal complement,
respectively and $A|\mathcal{M}$ denotes the restriction of $A$ to
$\mathcal{M}$.
\par \vskip 0.2 cm \noindent In this  section we introduce basic notations and recall some well-know concepts
of some classes of bounded and  unbounded operators in a Hilbert
space.\par\vskip 0.2 cm \noindent An operator $A \in
\mathcal{B}(\mathcal{H})$  is said to
 be:
normal if $A^*A=AA^*$ \big(  equivalently $\|Ax\|=\|A^*x\|$ for all
$x\in \mathcal{H}$\big),hyponormal if $A^*A\geq AA^*$ \big(
equivalently $\|Ax\|\geq\|A^*x\|$ for all $x\in \mathcal{H}\big),$
Co-hyponormal if $AA^*\geq A^*A$ \big(  equivalently
$\|A^*x\|\geq\|Ax\|$ for all $x\in \mathcal{H}\big),$ quasinormal if
$AA^*A=A^*A^2$ \big(  equivalently if $(A^*A)^2=A^{*2}A^2$\big) and
Quasi-hyponormal if $A^{*2}A^2 \geq (A^*A)^2$ \big( equivalently
$\|A^{2}x\|\geq \|A^*Ax\|$ for all $x\in
\mathcal{H}$\big).\par\vskip 0.2 cm \noindent An operator
$A:\mathcal{D}(A)\subset \mathcal{H} \longrightarrow \mathcal{H}$ is
said to be densely defined if $\mathcal{D}(A)$ (the domaine of $A$)
is dense in $\mathcal{H}$ and it is said to be  closed if its graph
is closed.
 The (Hilbert) adjoint of $A$ is denoted by $A^*$ and it is known
to be unique if A is densely defined.
 \par\vskip 0.2cm
\noindent We denote by ${Op}(\mathcal{H})$ the set of unbounded
densely defined linear operators on $\mathcal{H}$.\par \vskip 0.2 cm
\noindent Let $A,B\in {Op}(\mathcal{H})$, the product $AB$ of two
unbounded operators is defined by $$(AB)x =
A(Bx)\;\;\hbox{on}\;\;\mathcal{D}(AB) =\{\;x \in \mathcal{D}(B) : Bx
\in \mathcal{D}(A)\;\;\}.$$
 Let $A,B\in
{Op}(\mathcal{H})$ , we recall that $B$ is called an extension of
$A$, denoted by $A\subseteq B,$ if $\mathcal{D}(A)\subset
\mathcal{D}(B)$  and $Ax=Bx$ for all $x \in \mathcal{D}(A)$. An
closed operator $A \in {Op}(\mathcal{H})$ is said to commute with
$B\in \mathcal{B}(\mathcal{H})$ if $BA\subseteq AB$ , that is, if
for $x \in \mathcal{D}(A)$, we have $Bx \in \mathcal{D}(A)$ and
$BAx=ABx$.\par \vskip 0.2 cm \noindent Let $A\in
\mathcal{B}(\mathcal{H})$ and $ B\in Op(\mathcal{H})$ we say that
$A$ commutes with $B$ if $BA \subseteq AB$.\par \vskip 0.2 cm
\noindent \noindent Recall also that if $A$, $B$  and $AB$ are all
densely defined, then we have $A^*B^*\subseteq (BA)^*$. There are
cases where equality holds in the previous inclusion, namely if $B$
is bounded. For other notions and results about bounded and
unbounded operators, the reader may consult $[19].$\par\vskip 0.2 cm
\noindent A generalization of normal, quasinormal, hyponormal and
paranormal operators to unbounded normal quasinormal, hyponormal and
paranormal operators has been presented by several authors in the
last years. Some important references are $[ 7,9,10,11,20,21].$
\par \vskip 0.2 cm \noindent An operator $A\in Op(\mathcal{H})$, we said to be
hyponormal if $\mathcal{D}(A)\subset \mathcal{D}(A^*)$ and $\|A^*x\|
\leq \|Ax\|$ for all $x \in \mathcal{D}(A).$We refer to $[10]$ for
basic facts concerning unbounded hyponormal operators.\par\vskip 0.2
cm \noindent An operator $A\in Op(\mathcal{H})$ is said to be normal
if $A^*A = AA^*$. A closed operator  $A\in Op(\mathcal{H})$ is
normal if and only if $\mathcal{D}(A) =\mathcal{D}(A^*)$ and
$\|Ax\|=\|A^*x\|$ for all $x\in \mathcal{D}(A).$\par \vskip 0.2 cm
\noindent A densely defined operator $A:\mathcal{D}(A)\subseteq
\mathcal{H}\longrightarrow \mathcal{H}$ is said to be paranormal if
$$\|Ax\|^2\leq \|A^2x\|\|x\|\;\;\hbox{for all}\;\;x \in \mathcal{D}(A^2)$$ or equivalently
$$\|Ax\|^2\leq \|A^2x\|\;\;\hbox{for every unit vector }\;\;x \in
\mathcal{D}(A^2). \;\;(\hbox{See}\;\; [7]).$$
\par\vskip 0.2 cm \noindent Let $A$ and $B$ be
normal operators on a complex separable Hilbert space $\mathcal{H}.$
The equation $AX = XB$ implies $A^*X = XB^* $ for some operator
$X\in \mathcal{B}(\mathcal{H})$ is known as the familiar
Fuglede-Putnam theorem. $(\hbox{See}\;\; [14])$.\par\vskip
0.2cm\noindent
\par\vskip 0.2 cm \noindent Consider two normal (resp. hyponormal) operators $A$ and $B$ on a Hilbert space. It is known that, in
general, $AB$ is not normal (resp. not hyponormal). Kaplansky showed
that it may be possible that $AB$ is normal while $BA$ is not.
Indeed, he showed that if $A$ and $AB$ are normal, then $BA$ is
normal if and only if B commutes with $AA^*$, $(\hbox{see}\;
\;[12])$.\par\vskip 0.2 cm\noindent  In $[18, \hbox{Theorem}\;3 ]$,
Patel and Ramanujan proved that if $A$ and $B\in
\mathcal{B}(\mathcal{H})$ are hyponormal such that $A$ commutes with
$|B|$ and $B$ commutes with $|A^*|$  then $AB$ and $BA$ are
hyponormal.\par\vskip 0.2cm \noindent The study of operators
satisfying Kaplansky theorem  is of significant interest and is
currently being done by a number of mathematicians around the
world.Some developments toward this subject have been done in
$[6,10,12,15,16,17]$ and the references therein.\par \vskip 0.2 cm
\noindent The aim of this paper is to give sufficient conditions on
two some $(p,k)$-quasihyponormal operators (bounded or not), defined
on a Hilbert space, which make their product
$(p,k)$-quasihyponormal.
 The inspiration for our investigation comes
from $[1], [15]$ and $[17]$.\par \vskip 0.2 cm \noindent The outline
of the paper is as follows. First of all,we introduce notations and
consider a few preliminary results which are useful to prove the
main result. In the second section we discussed conditions which
ensure hyponormality,$k$-quasihyponormality or
$(p,k)$-quasihyponormality of the product of
hyponormal,$k$-quasihyponormal or $(p,k)$-quasihyponormality of
operators. In Section three, the concepts of $k$-quasihyponormal and
$k$-paranormal unbounded operators  are introduced.We give
sufficient conditions  which ensure $k$-quasihyponormality
($k$-parnormality or $k$-*-paranormality) of the product of
$k$-quasihyponormal ( $k$-paranotmal or $k-*$-paranormal) of
unbounded operators.
\section{ KAPLANSKY LIKE  THEOREM FOR BOUNDED $(p,k)$-QUASIHYPONORMAL OPERATORS}
The next definitions and lemmas give a brief description for the
background on which the paper will build on.
\begin{definition}
An operator $A\in \mathcal{B}(\mathcal{H})$ is said to be \par\vskip
0.2 cm \noindent (1)\;$p$-hyponormal if  $A^*A)^p-(AA^*)^p\geq0$ for
$0<p\leq 1$ \;  $ ([3] )$.\par \vskip 0.2 cm \noindent
(2)\;$p$-quasihyponormal  if $A^{*}\bigg(
(A^*A)^p-(AA^*)^p\bigg)A\geq 0,\;\; 0<p\leq 1$ \;($( [4] )$.\par
\vskip 0.2 cm \noindent (3)\;$k$-quasihyponormal operator if
$A^k(A^*A-AA^*)A^k\geq 0$ for positive integer $k$ $( [5])$.\par
\vskip 0.2 cm \noindent (4)\; $(p.k)$-quasihyponormal if
$A^{*k}\bigg( (A^*A)^p-(AA^*)^p\bigg)A^k\geq 0,\;0<p\leq
1\;\;\hbox{and} \; k\;\hbox{is positive integer}$\;$([13]).$
\par\vskip 0.2 cm \noindent
A $(p,k)$-quasihyponormal is an extension of $p$-hyponormal,
$p$-quasihyponormal and $k$-quasihyponormal.
\end{definition}

\begin{remark} Let $\mathcal{N},\mathcal{HN},p\mathcal{H},\mathcal{Q}(p)$ and $\mathcal{Q}(p,k)$ denote the classes consisting of
normal, hyponormal, $p$-hyponormal,$p$-quasihyponormal and
$(p,k)$-quasihyponormal operators.\par\vskip 0.2 cm \noindent These
classes are related by the proper inclusion. $( \hbox{See}\;
\;[14])$.
$$\mathcal{N}\subset\mathcal{HN}\subset p\mathcal{H}\subset\mathcal{Q}(p)\subset \mathcal{Q}(p,k).$$

\end{remark}
\noindent We need the following lemma which is important for the
sequel.
\begin{lemma}\par \vskip 0.2 cm \noindent
Let $A,B \in \mathcal{B}(\mathcal{H}).$ Then the following
properties hold \par \vskip 0.2 cm \noindent(1)\;If $A\geq B$ then
$C^*AC\geq C^*BC,\;\;\hbox{for all}\;\;C\in
\mathcal{B}(\mathcal{H}).$\par\vskip 0.2 cm \noindent (2)\;\; If
range of $C$ is dense in $\mathcal{H}$  then
$$A\geq B \Longleftrightarrow C^*AC\geq C^*BC.$$
\end{lemma}
\begin{proof}
This proof will be left to the reader.
\end{proof}
\noindent The following famous inequality is needful.
\begin{lemma}$( \;[2]\;\;\hbox{ Hansen's inequality}\;)$\par \vskip 0.2 cm \noindent
Let $A,B\in \mathcal{B}(\mathcal{H})$ such that $A\geq 0$ and $\|B\|
\leq 1$ then
$$\bigg(B^*AB\bigg)^\alpha \geq B^*A^\alpha B\;\;\;0< \alpha \leq 1.$$
\end{lemma}
 \noindent Kaplansky showed that it may be possible that $AB$ is normal
while $BA$ is not. Indeed, he showed that if $A$ and $AB$ are
normal, then $BA$ is normal if and only if $B$ commutes with $|A|$.
\par\vskip 0.2 cm \noindent  Kaplansky theorem\'{}s has been extended form normal operators
to hyponormal operators and unbounded hyponormal operators by the
authors in $[1].$ We collect some of their results in the following
theorem.
\begin{theorem}$(\;[1]\;)$
Let $A,B \in \mathcal{B}(\mathcal{H})$.The following statements
hold:
\par \vskip 0.2 cm \noindent (1)\; If $A$ is normal and $AB$ is
hyponormal then
$$BAA^*=AA^*B \Longrightarrow  BA\;\hbox{is hyponormal}.$$

\par \vskip 0.2 cm \noindent (2)\; If $A$ is normal and $AB$ is
co-hyponormal then
$$BAA^*=AA^*B \Longrightarrow  BA\;\hbox{is co-hyponormal}.$$

\par\vskip 0.2 cm \noindent (3)\;
 If $A$ is normal, $AB$ is
hyponormal and $BA$ is co-hyponormal then
$$BAA^*=AA^*B \Longleftrightarrow AB\;\hbox{and}\; BA\;\hbox{are normal}.$$
\end{theorem}
We give another proof of Kaplansky theorem.
 \begin{theorem}$[\hbox{Kaplansky}, [12])$
Let $A$ and $B \in \mathcal{B}(\mathcal{H})$ be two bounded
operators such that $AB$ and $A$ are normal. Then $$A^{*}AB=BA^{*}A
\Leftrightarrow(BA)\;\;\hbox{is normal}.\;$$
\end{theorem}
  \begin{proof}
$^{\prime\prime}\Longrightarrow ^{\prime\prime}$ \par \vskip 0.2 cm
\noindent Assume that   $A^{*}AB=BA^{*}A$  and we need to prove that
$BA$ is normal.\par \vskip 0.2 cm It is well know that $A$ is normal
if and only if $\left\|Ax\right\|=\left\|A^{*}x\right\|$ for all
$x\in \mathcal{H}$.\par \vskip 0.2 cm \noindent
  Since $AB$ is normal we have
  $$\|(AB)Ax\|=\|(AB)^*Ax\| \;\hbox{for all }\;\;\;x\in \mathcal{H} $$
  and we deduce that
  $$\|A(BA)x\|=\|A(BA)^*x\|\;\;\;\hbox{for all }\;x\in \mathcal{H}.$$
\par \vskip 0.3 cm \noindent
By hypotheses given in the theorem, we have \begin{eqnarray*}
 &&\left\|A\left(BA\right)x\right\|= \left\|A\left(BA\right)^{*}x\right\|\\&\Leftrightarrow& \langle A(BA)x,\;A(BA)x\rangle=\left\langle A(BA)^{*}x,\;A(BA^{*})x\right\rangle\\
& \Leftrightarrow&\left\langle \left[A(BA)\right]^{*}
A(BA)x,\;x\right\rangle=\left\langle
\left[A(BA^{*})\right]^{*}A(BA)^{*}x,\;x\right\rangle
\\&\Leftrightarrow& \left\langle \left[(BA)\right]^{*}A^{*}
A(BA)x,\;x\right\rangle=\left\langle
 \left[(BA)\right]A^{*}A(BA)^{*}x,\;x\right\rangle
\\& \Leftrightarrow&\left\langle A^{*}B^{*}A^{*}
A(BA)x,\;x\right\rangle =\left\langle
 BAA^{*}AA^{*}B^{*}x,\;x\right\rangle \\&
 \Leftrightarrow&\left\langle A^{*}AA^{*}B^{*}BAx,\;x\right\rangle=\left\langle
 AA^{*}AA^{*}BAA^{*}B^{*}x,\;x\right\rangle \\&
 \Leftrightarrow &\left\langle \left(BA\right)^{*}\left(BA\right)x,\;A^{*}Ax\right\rangle=\left\langle
 \left(BA\right)\left(BA\right)^{*}x,\;
 AA^{*}x\right\rangle\\&
  \Leftrightarrow &\langle\bigg((BA)^{*}(BA)-(BA)(BA)^{*}\bigg)x,\; A^*Ax\rangle=0\;\;\hbox{for all }\;x\in \mathcal{H}.\end{eqnarray*}
Put $T=\bigg((BA)^{*}(BA)-(BA)(BA)^{*}\bigg)$. \par \vskip 0.2 cm
\noindent Form the identities above we have $\langle
Tx,\;A^*Ax\rangle=0\;\;\;\hbox{for all }\;x\in \mathcal{H}.$ This
implies  that $$Tx \in
\mathcal{R}(A^*A)^\bot=\overline{\mathcal{R}(A^*A)}^\bot
=\mathcal{N}(A^*A)=\mathcal{N}(A)\;\;\;\hbox{for all }\;x\in
\mathcal{H}.$$
   \par \vskip 0.2 cm
  \noindent Now if $\mathcal{N}(A)=\{0\}$, we have $Tx=0\;\hbox{for all }\;x\in
  \mathcal{H}$ and $T\equiv 0$ i.e;. $BA$ is normal.\par \vskip 0.2
  cm \noindent If $\mathcal{N}(A)\not=\{0\}$. Suppose that,
contrary to our claim, the operator  $T\not\equiv
  0.$ There exists $x_0\in \mathcal{H}$,\;$x_0\not=0$ such that
  $Tx_0\not=0.$ Since $Tx_0 \in \mathcal{N}(A)$ and $\mathcal{H}=\mathcal{N}(A) \oplus
  \mathcal{N}(A)^\bot$, it follows that $Tx_0\notin \mathcal{N}(A)^\bot.
  $ From this we deduce that there exists $z_0 \in  \mathcal{N}(A)$
  so that $\langle Tx_0,\;z_0\rangle \not=0.$ \par \vskip 0.2 cm
  \noindent
As usual, this leads to the statement that
   $$0\not=\langle Tx_0,\;z_0\rangle=\langle x_0,\;Tz_0\rangle \Longrightarrow x_0\notin  \mathcal{N}(A)^\bot\;\;(\hbox{since}\;\;Tz_0 \in
  \mathcal{N}(A)).$$ This means that, $x_0 \in \mathcal{N}(A)$ and
  $$Tx_0=\bigg((BA)^{*}(BA)-(BA)(BA)^{*}\bigg)x_0=-(BA)(BA)^*x_0=0.$$ This
contradicts the assumption that $Tx_0 \not=0.$\par \vskip 0.2 cm
\noindent
  (2) $^{\prime \prime} \Longleftarrow  ^{\prime \prime}$
   The reverse application is even evidence of Kaplansky\par \vskip
   0.2 cm \noindent
  We have $ABA=ABA \Rightarrow (AB)A=A(BA)$, and by the theorem of Fuglede-Putnam
  $$(AB)^{*}A=A(BA)^{*} \Rightarrow ((AB)^{*}A)^{*}=(A(BA)^{*})^{*}.$$
  So
  $$A^{*}(AB)=(BA)A^{*} \Rightarrow A^{*}AB=BAA^{*}.$$
   \end{proof}
   \noindent
Consider two quasihyponormal operators $A$ and $B$ on a Hilbert
space. It is known that, in general, $AB$ is not quasihyponormal.
\begin{example} Let
$\mathcal{H}=l^2(\mathbb{N})$ with the canonical orthonormal basis
$(e_n)_{n \in \mathbb{N}}$ and consider $A$ the unilateral right
shift operator on $\mathcal{H}$ defined by $Ae_n=e_{n+1}$ for all $n
\in \mathbb{N}$ and $B$ the operator defined on $\mathcal{H}$ by
$$Be_n=\left\{
\begin{array}{ll}
 e_n ,\;\;\hbox{if}\;\;n\not=1\\
 \\
0 \;\hbox{if}\; n=1.\\
\end{array}
  \right.$$

A simple calculation shows that  $A$ and $B$ are quasihyponormal and
$AB$ does not quasihyponormal,since
$$\|AB)^{*}(AB)\big)e_0\|=1
\hbox{and}\;\;\|(AB)^2e_0\|=0.$$
\end{example}
\par \vskip 0.2 cm \noindent Denote by  $\mathbb{C}^{mn}$ the set of all $ m\times n$ complex
matrix.\par \vskip 0.2 cm \noindent In $[8]$,the authors proved the
following results.
\begin{theorem} $([8])$ Let $A=UH$,where $H\in \mathbb{C}^{nn}$ is
positive semidefinite
 Hermitian and $U\in \mathbb{C}^{nn}$ is unitary, and let $B\in
 \mathbb{C}^{nn},$ \par \vskip 0.2 cm \noindent (1)\; if $BU$ is
 normal and $HBU=BUH$, then $AB$ and $BA$ are normal,\par\vskip 0.2
 cm \noindent (2)\;if $AB$ and $BA$ are normal,then $HBU=BUH.$
\end{theorem}
\begin{theorem}$([8])$
Let $A \in \mathbb{C}^{mn}$ and $B\in \mathbb{C}^{nm}.$ Then $AB$
and $BA$ are normal if and only if $A^*AB=BA^*A$ and $ABB^*=BB^*A.$
\end{theorem}
\par \vskip
0.2 cm We show here the main results of this paper. Our intention is
to study
 some conditions for which the product of operators will
be hyponormal and $k$-quasihyponormal or $(k,p)$-quasihyponormal.
\begin{proposition}
Let $A=U|A|\in \mathcal{B}(\mathcal{H})$ with $U$ is unitary  and
let $B\in \mathcal{B}(\mathcal{H})$ such that $|A|BU=BU|A|.$ The
following properties hold \par \vskip 0.2 cm \noindent (1)\; If $UB$
is hyponormal, then $AB$ is hyponormal.\par \vskip 0.2 cm \noindent
(2)\; If $BU$ is hyponormal, then $BA$ is hyponormal.\par \vskip 0.2
cm \noindent (3)\;If $UB$ is quasihyponormal, then $AB$ is
quasihyponormal.\par \vskip 0.2 cm \noindent (4)\;\;If $BU$ is
quasihyponormal, then $BA$ is quasihyponormal.
\end{proposition}
\begin{proof}
(1) \;\;Suppose that $UB$ is hyponormal.Then \begin{eqnarray*}
\|(AB)^*x\|&=&\|B^*|A|U^*x\|\\&=&\|B^*U^*U|A|U^*x\| \\&=&
\|(UB)^*U|A|U^*x\|\\&\leq&\|UBU|A|U^*x\|\\&\leq&
\|U|A|BUU^*x\|\\&=&\|ABx\|.
\end{eqnarray*}
This shows that $AB$ is hyponormal. \par\vskip 0.2 cm \noindent
(2)\;\;Suppose that $BU$ is hyponormal.Then \begin{eqnarray*}
\|(BA)^*x\|&=&\||A|U^*B^*x\|\\&=&\||A|(BU)^*x\|
\\&=& \|(BU)^*|A|x\|\\&\leq&\|BU|A|x\|\\&=&\|BAx\|.
\end{eqnarray*}This shows that $BA$ is hyponormal.\par \vskip 0.2 cm
\noindent (3)\; Assume that $UB$ is quasihyponormal, then
\begin{eqnarray*}
\|(AB)^*(AB)x\|&=&\|B^*|A|^2Bx\|\\&\leq&\|B^*|A|^2BUU^*x\|\\&=&\|B^*BU|A|^2U^*x\|\\&=&\|B^*U^*UBU|A|^2U^*x\|\\&=&
\|(UB)^*(UB)U|A|^2U^*x\|\\&\leq&\|(UB)^2U|A|^2U^*x\|\;\;(\hbox{since}\;\;UB\;\;\hbox{is
quasihyponormal})\\&\leq&\|UBUBU|A|^2U^*x\|\\&\leq&
\|UBU|A|^2Bx\|\\&\leq&\| U|A|BU|A|Bx\|\\&=&\|(AB)^2x\|.
\end{eqnarray*}This shows that $AB$ is quasihyponormal.\par\vskip
0.2 cm \noindent (4)\;Assume that $BU$is quasihyponormal. Then
\begin{eqnarray*}
\|(BA)^*(BA)x\|&=&\|A^*B^*BAx\|\\&\leq&
\||A|(BU)^*BU|A|x\|\\&=&\|(BU)^*(BU)|A^2x)\|\\&\leq&\|(BU)^2|A|^2x\|\;\;(\hbox{since}\;\;BU\;\;\hbox{is
quasihyponormal})\\&\leq&
\|BUBU|A|^2x\|\\&\leq&\|BU|A|.BU|A|x\|\\&=&\|(BA)^2x\|.
\end{eqnarray*}
This shows that $BA$ is quasihyponormal.
\end{proof}
\begin{proposition}
Let $A$ and $B \in \mathcal{B}(\mathcal{H})$ are hyponormal
operators.   If $BA^*=A^*B$, then $AB$  and $BA$ are
$k$-quasihyponormal.
\end{proposition}
\begin{proof}
  Let $x \in \mathcal{H}$,
     \begin{eqnarray*}
     \left\|(AB)^{*}(AB)^kx\right\| &=&
     \left\|B^{*}A^{*}(AB)^kx\right\|\\&
                                  \leq& \left\|BA^{*}(AB)^kx\right\| ( \hbox{since}\;\;B\;\hbox{is hyponormal})\\
                                  & \leq& \left\|A^{*}B(AB)^kx\right\|\\
                                    &\leq& \left\|AB(AB)^kx\right\| ( \hbox{since}\;\;A\;\hbox{is hyponormal})\\
                                     &\leq&
                                     \left\|\left(AB\right)^{k+1}x\right\|\end{eqnarray*}
                                     we even have evidence to
                                    $k$- quasi-hyponormal.

\end{proof}

\begin{proposition}
Let $A$ and $B \in \mathcal{B}(\mathcal{H})$  such that $A$ and $B$
 are doubly commutative $k$-quasi-hyponormal operators then $AB$ is
$k$- quasi-hyponormal.
\end{proposition}
\begin{proof}
Let $x\in \mathcal{H},$
$$\|(AB)^*(AB)^kx\|=\|A^*A^kB^kB^*x\|\leq \|A^{k+1}B^*B^kx\|\leq \|A^{k+1}B^{k+1}x\|=\|(AB)^{k+1}x\|$$
\end{proof}
\begin{proposition}
Let $A$ and $B\in \mathcal{B}(\mathcal{H})$ are $k$-quasihyponormal
operators for some positive integer $k$. The following statements
hold\par\vskip 0.2 cm \noindent (1)\;If $A^*A^kB=BA^*A^k$ and
$A^jB^j=(AB)^j$  for $j\in \{k,k+1\}$,then $AB$ is $k$-
quashyponormal.
\par\vskip 0.2 cm \noindent (2)\;If $B^*B^kA=AB^*B^k$ and
$B^jA^j=(BA)^j$  for $j\in \{k,k+1\}$,then $BA$ is $k$-
quashyponormal.
\end{proposition}
\begin{proof}
(1)
\begin{eqnarray*}\|(AB)^*(AB)^kx\|&=&\|B^*A^*A^kB^kx\|\\&=&\|B^*B^kA^*A^kx
\|\\&\leq&\| B^{k+1}A^*A^kx \|\;\;\;(\hbox
{since}\;\;B\;\;\hbox{is}\;
k-\hbox{quasihyponormal}\\&\leq&\|A^*A^kB^{k+1}\|\\&\leq&\|
A^{k+1}B^{k+1}x\|\quad(\hbox{since}\;A\;\hbox{is}\;k-\hbox{
quasighyponormal})\\&\leq& \|(AB)^{k+1}x\|\;\quad \hbox{for
all}\;x\in \mathcal{H}.
\end{eqnarray*}
\noindent (2)
\begin{eqnarray*}\|(BA)^*(BA)^kx\|&=&\|A^*B^*B^kA^kx\|\\&=&\|A^*A^kB^*B^kx
\|\\&\leq&\| A^{k+1}B^*B^kx \|\;\;\;(\hbox
{since}\;\;A\;\;\hbox{is}\;k-\hbox{
quasihyponormal}\\&\leq&\|B^*B^kA^{k+1}x\|\\&\leq&\|
B^{k+1}A^{k+1}x\|\;\quad(\hbox{since}\;B\;\hbox{is}\;k-\hbox{
quasihyponormal})\\&=& \|(BA)^{k+1}x\|\;\quad\hbox{for all}\;x\in
\mathcal{H}.
\end{eqnarray*}
\end{proof}

\begin{proposition}
Let $A$ and $B \in \mathcal{B}(\mathcal{H})$ such that $A$ is normal
and $AB$ is quasinormal,then
$$A^*AB=BA^*A\Longrightarrow  BA\;\;\hbox{is quasinormal}.$$
\end{proposition}
\begin{proof}
Let $A=U|A|$ with $U$ unitary. Since $A^*AB=BA^*A$ we have
$|A|B=B|A|.$ These facts and the quasinormality of $AB$ give
\begin{eqnarray*}
(BA)(BA)^*(BA)&=&U^*(AB)U\big(U^*(AB)U\big)^*(U^*(AB)U)\\&=&U^*(AB)(AB)^*(AB)U\\&=&
U^*(AB)^*(AB)^2U\\&=&(U^*(AB)U)^*(U^*(AB)U)^2\\&=&(BA)^*(BA)^2.
\end{eqnarray*}
\end{proof}

\begin{remark}
The reverse implication does not hold in the previous result (even
if $A$ is self-adjoint) as shown in the following
example\end{remark}\par\vskip 0.2 cm \noindent
\begin{example}
Let $A$ and $B$ be acting on the standard basis $(e_n)$ of
$\ell^2(\N)$ by: $$Ae_n=\alpha_ne_n \;\;\hbox{and}\;\;
Be_n=e_{n+1},~\forall n\geq 1$$ respectively. Assume further that
$\alpha_n$ is bounded, \textit{real-valued} and \textit{positive},
for all $n$. Hence $A$ is self-adjoint (hence normal!) and positive.
Then $$ABe_n=\alpha_{n+1}e_{n+1},~\forall n\geq 1.$$ For
convenience, let us carry out the calculations as infinite matrices.
Then $$
 AB=\left(\begin{matrix} 0 & 0 &  &  & &\text{\Large{0}} \\ \alpha_1 & 0 &
0 &  &
\\ 0 & \alpha_2 & 0 & 0 &
 \\
 & 0 & \alpha_3 & 0 & \ddots \\
 &  & 0 & \ddots & 0 & \\
\text{\Large{0}} &  &  & \ddots & \ddots & \ddots
\end{matrix}\right) \text{ so that } (AB)^*=\left(\begin{matrix} 0 & \alpha_1 &  &  & &\text{\Large{0}} \\ 0 & 0 &
\alpha_2 &  &
\\ 0 & 0 & 0 & \alpha_3 &
 \\
 & 0 & 0 & 0 & \ddots \\
 &  & 0 & \ddots & 0 & \\
\text{\Large{0}} &  &  & \ddots & \ddots & \ddots
\end{matrix}\right).$$
Hence $$(AB)^2=\left(\begin{matrix} 0 & 0 &  &  & &\text{\Large{0}} \\
0 & 0 & 0 &  &
\\  \alpha_1\alpha_2 & 0 & 0 &
 \\
 0 & \alpha_2\alpha_3 & 0 & \ddots \\
 &  & 0 & \ddots & 0 & \\
\text{\Large{0}} &  &  & \ddots & \ddots & \ddots
\end{matrix}\right) \text{ and }\left[(AB)^*\right]^2=\left(\begin{matrix} 0 & 0 &\alpha_1\alpha_2 &  &  & &\text{\Large{0}} \\ 0 & 0 &
0 & \alpha_2\alpha_3 &
\\ 0 & 0 & 0 & \ &
 \\
 & 0 & 0 & 0 & \ddots \\
 &  & 0 & \ddots & 0 & \\
\text{\Large{0}} &  &  & \ddots & \ddots & \ddots
 \end{matrix}\right).$$
 so
$$(AB)^*AB=\left(\begin{matrix} \alpha_1^2 & 0 &  &  & &\text{\Large{0}} \\
0 &\alpha_2^2 & 0 &  &
\\ 0 & 0 & \alpha_3^2 & 0 &
 \\
 & 0 & 0 & \ddots & \ddots \\
 &  & 0 & \ddots & \ddots & \\
\text{\Large{0}} &  &  & \ddots & \ddots & \ddots
\end{matrix}\right).$$  { this implies } $$ \left[(AB)^*AB\right]^2=\left(\begin{matrix} \alpha_1^4 & 0 &  &  & &\text{\Large{0}} \\ 0
&\alpha_2^4 & 0 &  &
\\ 0 & 0 & \alpha_3^4 & 0 &
 \\
 & 0 & 0 & \ddots & \ddots \\
 &  & 0 & \ddots & \ddots & \\
\text{\Large{0}} &  &  & \ddots & \ddots & \ddots
\end{matrix}\right).$${ and }\\

$$
\left[(AB)^*\right]^2\left[AB\right]^2=\left(\begin{matrix}
\left(\alpha_1\alpha_2\right)^2 & 0 &  &  & &\text{\Large{0}} \\ 0
&\left(\alpha_2\alpha_3\right)^2 & 0 &  &
\\ 0 & 0 & \left(\left(\alpha_3\alpha_4\right)\right)^2 & 0 &
 \\
 & 0 & 0 & \ddots & \ddots \\
 &  & 0 & \ddots & \ddots & \\
\text{\Large{0}} &  &  & \ddots & \ddots & \ddots
\end{matrix}\right).$$
It thus becomes clear that $AB$ is quasi hyponormal iff
$\alpha_n\leq \alpha_{n+1}$.

Similarly
\[BAe_n=\alpha_{n}e_{n+1},~\forall n\geq 1.\]
Whence the matrix representing $BA$ is given by:
\[BA=\left(\begin{matrix} 0 & 0 &  &  & &\text{\Large{0}} \\ \alpha_2 & 0 &
0 &  &
\\ 0 & \alpha_3 & 0 & 0 &
 \\
 & 0 & \alpha_4 & 0 & \ddots \\
 &  & 0 & \ddots & 0 & \\
\text{\Large{0}} &  &  & \ddots & \ddots & \ddots
\end{matrix} \right)\text{ so that } (BA)^*=\left(\begin{matrix} 0 & \alpha_2 &  &  & &\text{\Large{0}} \\ 0 & 0 &
\alpha_3 &  &
\\ 0 & 0 & 0 & \alpha_4 &
 \\
 & 0 & 0 & 0 & \ddots \\
 &  & 0 & \ddots & 0 & \\
\text{\Large{0}} &  &  & \ddots & \ddots & \ddots
\end{matrix}\right).\]
Therefore,
\[\left(BA\right)^2=\left(\begin{matrix} 0 & 0 &  &  & &\text{\Large{0}} \\ 0 & 0 &
0 &  &
\\  \alpha_2\alpha_3 & 0 & 0 &
 \\
 0 & \alpha_3\alpha_4 & 0 & \ddots \\
 &  & 0 & \ddots & 0 & \\
\text{\Large{0}} &  &  & \ddots & \ddots & \ddots
\end{matrix}\right),
\left[(BA)^*\right]^2=\left(\begin{matrix} 0 & 0 &\alpha_2\alpha_3 &
& & &\text{\Large{0}} \\ 0 & 0 & 0 & \alpha_3\alpha_4 &
\\ 0 & 0 & 0 & \ &
 \\
 & 0 & 0 & 0 & \ddots \\
 &  & 0 & \ddots & 0 & \\
\text{\Large{0}} &  &  & \ddots & \ddots & \ddots
 \end{matrix}\right).\]{and}
\[(BA)^*BA=\left(\begin{matrix} \alpha_2^2 & 0 &  &  & &\text{\Large{0}} \\ 0
&\alpha_3^2 & 0 &  &
\\ 0 & 0 & \alpha_4^2 & 0 &
 \\
 & 0 & 0 & \ddots & \ddots \\
 &  & 0 & \ddots & \ddots & \\
\text{\Large{0}} &  &  & \ddots & \ddots & \ddots
\end{matrix}\right).\] \\this implies\\$$
\left[(BA)^*BA\right]^2=\left(\begin{matrix} \alpha_2^4 & 0 &  &  &
&\text{\Large{0}} \\ 0 &\alpha_3^4 & 0 &  &
\\ 0 & 0 & \alpha_4^4 & 0 &
 \\
 & 0 & 0 & \ddots & \ddots \\
 &  & 0 & \ddots & \ddots & \\
\text{\Large{0}} &  &  & \ddots & \ddots & \ddots
\end{matrix}\right).$$ and
$$\left[(BA)^*\right]^2\left[BA\right]^2=\left(\begin{matrix}
\left(\alpha_2\alpha_3\right)^2 & 0 &  &  & &\text{\Large{0}} \\ 0
&\left(\alpha_3\alpha_4\right)^2 & 0 &  &
\\ 0 & 0 & \left(\left(\alpha_4\alpha_5\right)\right)^2 & 0 &
 \\
 & 0 & 0 & \ddots & \ddots \\
 &  & 0 & \ddots & \ddots & \\
\text{\Large{0}} &  &  & \ddots & \ddots & \ddots
\end{matrix}\right).$$
Accordingly, $BA$ is quasi hyponormal if and only if  $\alpha_n\leq
\alpha_{n+1}$ (thankfully, this is the same condition for the
hyponormality of $AB$).

Finally,

$$ BA^2=\left(\begin{matrix} 0 & 0 &  &  & &\text{\Large{0}} \\ \alpha_1^2
& 0 & 0 &  &
\\ 0 & \alpha_2^2 & 0 & 0 &
 \\
 & 0 & \alpha_3^2 & 0 & \ddots \\
 &  & 0 & \ddots & 0 & \\
\text{\Large{0}} &  &  & \ddots & \ddots & \ddots
\end{matrix}\right)\neq A^2B=\left(\begin{matrix} 0 & 0 &  &  & &\text{\Large{0}} \\ \alpha_2^2 & 0 &
0 &  &
\\ 0 & \alpha_3^2 & 0 & 0 &
 \\
 & 0 & \alpha_4^2 & 0 & \ddots \\
 &  & 0 & \ddots & 0 & \\
\text{\Large{0}} &  &  & \ddots & \ddots & \ddots
\end{matrix}\right).$$

\end{example}

\begin{proposition}
Let $A=U|A|$ and $B \in \mathcal{B}(\mathcal{H})$ such that $A$ and
$B$ are quasinormal.If $BU$ is quasinormal and $|A|BU=BU|A|$,then
$BA$ is quasinormal.
\end{proposition}
\begin{proof}
Let $A=U|A|$   be the polar decomposition of A with $U$ partial
isometry. Since $A$ is quasinormal we have $|A|U=U|A|.$ These facts
and the quasinormality of $AB$ give
\begin{eqnarray*}
(BA)(BA)^*(BA)&=&(B|A|U)\big(B|A|U\big)^*(B|A|U)\\&=&BU|A|^2(BU)^*BU|A|\\&=&
|A|^2(BU)(BU)^*(BU)|A|\\&=&|A|^2(BU)^*(BU)^2|A|\\&=&(BU|A|)^*(BU|A|)^2\\&=&(BA)^*(BA)^2.
\end{eqnarray*}
Thus $BA$ is quasinormal.
\end{proof}
\begin{proposition}
Let $A,B\in \mathcal{B}(\mathcal{H}) $ such that $A$ is normal and
$AB$ is paranormal.  Then $$A^*AB=BA^*A \Longrightarrow BA\;\;\hbox{
is paranormal}.$$
\end{proposition}
\begin{proof}
Let $A=U|A|$ with $U$ is unitary. Since $A$ is normal we have
$|A|U=U|A|$ and  hence $B|A|=|A|B$.This then gives  $$BA=U^*ABU.$$
From this fact we obtain that for all  unit vector $x\in
\mathcal{H},$
\begin{eqnarray*}
\|(BA)x\|^2&=&
\|U^*(AB)Ux\|^2\\&\leq&\|(AB)Ux\|^2\\&\leq&\|(AB)^2Ux\|\;\;(\hbox{since}\;\;AB\;\;\hbox{is
paranormal and }\;\;
\|Ux\|=1)\\&=&\|U\big(U^*ABU\big)^2x\|\\&\leq&\|(BA)^2x\|.
\end{eqnarray*}
Hence $BA$ is paranormal operator.
\end{proof}

Let $A$ and $B\in \mathcal{B}(\mathcal{H})$. The commutator of $A$
and $B$ is defined as $[A,B]=AB-BA$ and the self-commutator of $A$
is $[A^*,A]$.\par\vskip 0.2 cm \noindent The span of $A$ and $B$ is
$$span\{A,\;B\;\}:=\{\;aA+bB,\;\;\;a,b\in \mathbb{C}\;\;\;\}.$$

\begin{proposition}
Let $A$ and $B\in \mathcal{B}(\mathcal{H})$ such that $A$ is
quasihyponormal and $B$ is hyponormal. If $$[B^*,\;A]=A^*[B^*,B]B
=B^*[A^*,A]B=0.$$ Then $T=\omega A+B$ is quasihyponormal for all
$\omega \in \mathbb{C}.$
\end{proposition}
\begin{proof}
Note that if $\omega \in \mathbb{C}$ and $T=\omega A+B$ , then
$$[T^*,\;T]=|\omega|^2 [A^*,A]+[B^*,B]+2Re\big(\omega[B^*,A]
\big).$$ By hypotheses given in the theorem, we have
$$T^*[T^*,T]T=|\omega|^4\underbrace{A^*[A^*,A]A}_{\geq 0}+|\omega|^2\underbrace{A^*[B^*,B]A}_{\geq 0}+\underbrace{B^*[B^*,B]B}_{\geq 0}.$$
By Lemma 2.1 we get $T^*[T^*,T]T\geq 0.$ This completes the proof.
\end{proof}
\begin{proposition}
Let $A$ and $B\in \mathcal{B}(\mathcal{H})$ such that $A$  and $B$
are  quasihyponormal. If
$$[B^*,\;A]=A^*[B^*,B]B =B^*[A^*,A]B=A^*[B^*,B]A=0.$$ Then $T=\omega A+B$ is
quasihyponormal for all $\omega \in \mathbb{C}.$
\end{proposition}
\begin{proof}
By the same arguments as in the proof of the proposition above, we
have $$T^*[T^*,T]T=|\omega|^4\underbrace{A^*[A^*,A]A}_{\geq
0}+\underbrace{B^*[B^*,B]B}_{\geq 0}.$$
\end{proof}
The next result is a necessary and sufficient condition for
$span\{A,B\}$ to be quasihyponormal.
\begin{theorem}
Let $A$ and $B\in \mathcal{B}(\mathcal{H})$ such that
$$[B^*,\;A]=A^*[B^*,B]B =B^*[A^*,A]B=A^*[B^*,B]A=0.$$ Then $A$ and
$B$ are quasihyponormal if and only if $span\{A,B\}$ is
quasihyponormal.
\end{theorem}
\begin{proof}
First assume that $A$ and $B$ are quasihyponormal.It is immediate
from the preceding proposition that $span\{A,B\;\}$ is
quasihyponormal. The converse is immediate from the definition of
$span\{A,B\}.$
\end{proof}
\begin{theorem}
Let $A$ and $B \in \mathcal{B}(\mathcal{H})$ such that $A$ is normal
and $AB$ is $(p,k)$-quasihyponormal. If $B(AA^*)=(AA^*B)$ then $BA$
is $(p,k)$-quasihyponormal  for $0<p\leq 1$ and $k\in
\mathbb{{Z}}_+$.
\end{theorem}
\begin{proof}
Since $A$ is normal, we have $A=PU=UP$ with $P\geq 0$ and $U$
unitary.\par \vskip 0.2 cm \noindent$$B(AA^*)=(AA^*B)
P^2B=BP^2\Longrightarrow PB=BP.$$ A simple computation shows that
$$U^*ABU=BA.$$
since $(AB)^*(AB)$ and   $(AB)(AB)^*$ are positive  and $\|U\|\leq
1$, by using Lemma 2.2, it follows that
\begin{eqnarray*}
(BA)^k\bigg((BA)^*(BA)\bigg)^p(BA)^k&=&\big(U^*(AB)^*U\big)^k\bigg(
U^*(AB)^*(AB)U)\bigg)^p(U^*(AB)U)^k
\\&=& U^*(AB)^{*k}U\bigg(
U^*(AB)^*(AB)U\bigg)^pU^*(AB)^kU\\&\geq&
U^*(AB)^{*k}UU^*\bigg((AB)^*(AB)\bigg)^pUU^*(AB)^kU\\&\geq&U^*\underbrace{(AB)^{*k}\bigg((AB)^*(AB)\bigg)^p(AB)^k}U
\\&\geq&
U^*(AB)^{*k}\bigg((AB)(AB)^*\bigg)^p(AB)^kU\;\;(\hbox{since}\;
AB\;\hbox{is in}\;\; \mathcal{Q}(p,k))\\&\geq& U^*(AB)^{*k}\bigg(
U\underbrace{U^*(AB)UU^*(AB)^*U}_{\geq
0}U^*\bigg)^p(AB)^kU\\&\geq&U^*(AB)^{*k}U\bigg(  U^*(AB)UU^*(AB)^*U
\bigg)^pU^*(AB)^kU \;\;(\hbox{by Lemma 2.2})\\&\geq&
\big(U^*(AB)^*U\big)^k\bigg(  U^*(AB)UU^*(AB)^*U
\bigg)^p\big(U^*(AB)U\big)^k\\&=&(BA)^{*k}\bigg((BA)(BA)^*\bigg)^p(BA)^k.
\end{eqnarray*}
This implies that $BA$ is $(p,k)$-quasihyponormal operator.The proof
of this theorem is finished
\end{proof}

\begin{theorem}
Let $A$ and $B\in \mathcal{B}(\mathcal{H})$ such that $A$ is
$(p,k)$-quasihyponormal and $B$ is invertible. If $A$ commute with
$B$ and $B^*$ then $AB$ is $(p,k) $-quasihyponormal for $0<p\leq 1$
and $k\geq 1.$\end{theorem}
\begin{proof}
\begin{eqnarray*}
(AB)^{*k}\bigg( (AB)^*(AB)\bigg)^p(AB)^k&=&(AB)^{*k}\bigg(
B^*A^*AB\bigg)^p(AB)^k\\&\geq&(AB)^{*k}B^*(A^*A)^pB(AB)^k\\&\geq&
(B^*)^{k+1}\underbrace{A^{*k}(A^*A)^pA^k}B^{k+1}\\&\geq&(B^*)^{k+1}A^{*k}(AA^*)^pA^kB^{k+1}\\&\geq&(AB)^{*k}\bigg((AB)(AB)^*
\bigg)^p(AB)^k
\end{eqnarray*}
\end{proof}

\section{ KAPLANSKY LIKE THEOREM FOR UNBOUNDED $k$-QUASI-HYPONORMAL OPERATORS}
 In this
section, we generalized some notions of bounded operators to
unbounded operators on a Hilbert space and we give sufficient
conditions  which ensure $k$-quasihyponormality ($k$-parnormality or
$k$-*-paranormality) of the product of $k$-quasihyponormal (
$k$-paranotmal or $k-*$-paranormal) of  unbounded operators.
\par \vskip 0.2 cm \noindent For an $A\in Op(\mathcal{H})$, define
$A^2$ by
$$D(A^2)=\{\;x\in \mathcal{D}(A)/\;Ax\in
\mathcal{D}(A)\;\;\},A^2x=A(Ax).$$ We can define higher powers
recursively. Given $A^n$, define
 $$D(A^{n+1})=\{\;x\in
\mathcal{D}(A)/\;Ax\in \mathcal{D}(A^n)\;\;\},A^{n+1}x=A^n(Ax).$$Let
us begin with the concept of $k$-quasihyponormality.
\begin{definition}
A densely defined operator $A:\mathcal{D}(A) \subset
\mathcal{H}\longrightarrow \mathcal{H}$ is said to be
$k$-quasihyponormal  for some positive integer $k$ if
$\mathcal{D}(A)\subset \mathcal{D}(A^*)$ and $$ \|A^*A^kx\|\leq
\|A^{k+1}x\|\;\;\hbox{for all}\;\;x\in \mathcal{D}(A^{k+1}).$$

\end{definition}

\begin{remark}
 \par\vskip 0.2cm \noindent (1)\;
Clearly, the class of all $k$-quasihyponormal operators on
$\mathcal{H}$ contains all hyponormal operators.
\par \vskip 0.2 cm \noindent (2)\; the class of $k$-quasihyponormal
operators properly contains the classes of
$k^\prime$-quasihyponormal ($k^{\prime} < k$).
\end{remark}

\begin{definition}
A densely defined operator $A:\mathcal{D}(A)\subset
\mathcal{H}\longrightarrow \mathcal{H}$ is said to be
\par \vskip 0.2 cm \noindent (1)\;$k$-paranormal  for some positive integer $k$ if
$$\|Ax\|^k\leq \|A^kx\|\|x\|^{k-1}\;\;\hbox{for all}\;\;x \in
\mathcal{D}(A^k),$$ or equivalently
$$\|Ax\|^k\leq \|A^kx\|\;\;\hbox{for every unit vector }\;\;x \in
\mathcal{D}(A^k).$$
\par \vskip 0.2 cm \noindent (2)\;
$k-*$-paranormal for some positive integer $k$  if
$\mathcal{D}(A)\subset \mathcal{D}(A^*)$ and
$$\|A^*x\|^k\leq \|A^kx\|\|x\|^{k-1}\;\;\hbox{for all}\;\;x \in
\mathcal{D}(A^k),$$ or equivalently $\mathcal{D}(A)\subset
\mathcal{D}(A^*)$ and
$$\|A^*x\|^k\leq \|A^kx\|\;\;\hbox{for every unit vector }\;\;x \in
\mathcal{D}(A^k).$$
\end{definition}
\begin{example}
Consider the Hilbert space $\mathcal{H}=l^2(\mathbb{Z})$ under the
inner product $\langle
x,\;y\rangle=\displaystyle\sum_{n=-\infty}^\infty
x_n\overline{y_n}$, and let $(e_n)_{n\in \mathbb{Z}}$ be any
orthonormal basis for $\mathcal{H}.$ Let $(\omega_n)_{n \in
\mathbb{Z}}$ be a increasing sequence of numbers such that
$\omega_n>0$ for all $n\in \mathbb{Z}$ and
$\displaystyle\sup_n(\omega_n)=\infty.$ Consider  the unilateral
forward  weighted  shift operator $A$  defined in term of the
standard basis of $l^2(\mathbb{Z})$ by $$Ae_n=\omega_n e_{n+1}
\;\;\hbox{for all}\;\;n\in \mathbb{Z}.$$\par \vskip 0.2 cm \noindent
A simple calculation shows that the adjoint of unilateral forward
weighted shift is given by
$$A^*e_n=\omega_ne_{n-1}\;\;\hbox{for all}\;\;n\in \mathbb{Z}.$$
By this we have
$$A^*Ae_n=\omega_n^2e_n\;\;\hbox{and}\;\;A^2e_n=\omega_n\omega_{n+1}e_{n+2}.$$
Consequently $$\|A^*Ae_n\|\leq \|A^2e_n\|\;\;\;\hbox{for
all}\;\;n\in \mathbb{Z}.$$ \noindent Which implies that the operator
$A$ is quasihyponormal.
\end{example}
\begin{lemma}
If $A \in Op(\mathcal{H})$ be a $k$-quasihyponormal, then
$$\|A^kx\|^2\leq \|A^{k+1}x\|\|A^{k-1}x\|\;\;\hbox{for}\;\;x\in \mathcal{D}(A^{k+1}).$$
\end{lemma}
\begin{proof}
In fact \begin{eqnarray*}\|A^kx\|^2=\langle
A^kx\;,A^kx\rangle&=&\langle A^*A^kx\;,\;A^{k-1}x\rangle\\&\leq&
\|A^*A^{k}x\|\|A^{k-1}x\|\\&\leq& \|A^{k+1}x\|\|A^{k-1}x\|\;
.\end{eqnarray*}
\end{proof}
\begin{proposition}
Let $A$ be a closed densely defined operator in $\mathcal{H}$. If
$A$ is $k$-quasihyponormal then
$$A^{*k}\bigg( A^{*2}A^2-aA^*A+a^2I\bigg)A^k\geq 0\;\;\;\hbox{for all}\;\;a\in \mathbb{R}.$$
\end{proposition}
\begin{proof}
Let us suppose that $A$ is $k$-quasihyponormal. Then it follows that
the following relation holds:
$$\|A^*A^kx\|^2\leq \|A^{k+1}x\|^2\leq  \|A^{k+2}x\|\|A^{k}x\|.$$
This means that
\begin{eqnarray*}
\|A^*A^kx\|^2 \leq \|A^{k+2}x\|\|A^{k}x\| &\Longleftrightarrow&
4\|A^*A^kx\|^2 -4\|A^{k+2}x\|\|A^{k}x\|\leq 0\\&\Longleftrightarrow&
\|A^{k+2}x\|^2-2a\|A^{k+2}x\|^2+a^2\|A^kx\|^2\geq 0
\\&\Longleftrightarrow&A^{*k}\bigg( A^{*2}A^2-aA^*A+a^2I\bigg)A^k\geq 0 \;\;\;.
\end{eqnarray*}
This completes the proof of the proposition.
\end{proof}
\begin{proposition}
Let $A \in Op(\mathcal{H})$ be a $k$-quasihyponormal. if $A$ is
invertible then $A$ is hyponormal.
 \end{proposition}
 \begin{proof}
As $A$ is $k$-quasihyponormal,we have by definition that
$\mathcal{D}(A)\subset \mathcal{D}(A^*)$ and
$$\|A^*A^kx\|\leq \|A^{k+1}x\|\;\;\;\;\hbox{for all}\;\;x\in
\mathcal{D}(A^{k+1}).$$ Since $A$ is invertible with an everywhere
defined bounded inverse, we have for all $x \in \mathcal{D}(A):
A^{-k}x \in \mathcal{D}(A^{k+1}) $
$$\|A^*A^kA^{-k}x\|\leq
\|A^{k+1}A^{-k}x\|\;\;\;\;\hbox{for all}\;\;x\in
\mathcal{D}(A^{k+1}).$$ Hence we my write $$\|A^*x\| \leq \|Ax\|.$$
 \end{proof}
 \begin{proposition}
Let $A \in \mathcal{B}(\mathcal{H})$ and $B \in Op(\mathcal{H})$
such that $A$ and $B$ are hyponormal. If $BA^*\subseteq A^*B$ then
$AB$ is $k$- quasihyponormal.
\end{proposition}
\begin{proof}
     Let $x \in \mathcal{D}((AB)^{k+1})$,
     \begin{eqnarray*}
     \left\|(AB)^{*}(AB)^kx\right\| &=&
     \left\|B^{*}A^{*}(AB)^kx\right\|\\&
                                  \leq& \left\|BA^{*}(AB)^kx\right\| ( \hbox{since}\;\;B\;\hbox{is hyponormal})\\
                                  & \leq& \left\|A^{*}B(AB)^kx\right\|\\
                                    &\leq& \left\|AB(AB)^kx\right\| ( \hbox{since}\;\;A\;\hbox{is hyponormal})\\
                                     &\leq&
                                     \left\|\left(AB\right)^{k+1}x\right\|\end{eqnarray*}
                                     we even have evidence to quasi-hyponormal.
\end{proof}
\begin{proposition}
Let $A \in \mathcal{B}(\mathcal{H})$ and $B \in Op(\mathcal{H})$
 such that $A$ is normal and $AB$ is
$k$-quasihyponormal. If $B(AA^*)\subseteq (AA^*B)$ then $BA$ is
$k$-quasihyponormal.
\end{proposition}
\begin{proof}

Since $A$ is normal, we have $A=PU=UP$ with $P\geq 0$ and $U$
unitary.\par \vskip 0.2 cm \noindent$$B(AA^*)=(AA^*B)\Rightarrow
P^2B=BP^2\Longrightarrow PB=BP.$$ A simple computation shows that
$$U^*ABU=BA.$$
So
\begin{eqnarray*}\left\|\left(BA\right)^*\left(BA\right)^kx\right\|&=&\left\|U^*\left(AB\right)^*UU^*\left(AB\right)^kUx\right\|\\&=&\left\|U^*\left(AB\right)^*\left(AB\right)^kUx\right\|\\&\leq&
\left\|U^*\left(AB\right)^{k+1}Ux\right\|\\
&\leq& \left\|\big(U^*ABU\big)^{k+1}x\right\|\\&=&
\left\|\left(BA\right)^{k+1}x\right\|,\end{eqnarray*} This implies
that
    $\left(BA\right)$ is $k$-quasi hyponormal.

\end{proof}
\begin{proposition}

Let $A\in \mathcal{B}(\mathcal{H})$ and $B:D(B)\subset
\mathcal{H}\longrightarrow \mathcal{H}$ be closed densely defined
operator such that  $A$ and $B$ are $k$-quasihyponormal.  If
$A^*A^kB\subseteq BA^*A^k$ and $A^jB^j\subseteq (AB)^j$  for $j\in
\{k,k+1\}$,then $AB$ is $k$-quashyponormal.
\end{proposition}
\begin{proof}
Since $A\in \mathcal{B}(\mathcal{H})$ and $B$ is  closed densely
defined, it is well known that $$(AB)^*=B^*A^*.$$ Hence we may write
\begin{eqnarray*}\|(AB)^*(AB)^kx\|&=&\|B^*A^*A^kB^kx\|\\&=&\|B^*B^kA^*A^kx
\|\\&\leq&\| B^{k+1}A^*A^kx \|\;\;\;(\hbox {since}\;\;B\;\;\hbox{is
$k$-quasihyponormal}\\&\leq&\|A^*A^kB^{k+1}\|\\&\leq&\|
A^{k+1}B^{k+1}x\|\quad(\hbox{since}\;A\;\hbox{is}\;k-\hbox{
quasihyponormal})\\&\leq& \|(AB)^{k+1}x\|.
\end{eqnarray*}
This completes the proof.
\end{proof}

\begin{proposition}
Let $A\in Op(\mathcal{H})$ is normal and $B\in
\mathcal{B}(\mathcal{H}).$ \par \vskip 0.2 cm \noindent (1) \;If
$AB$ is $k$-paranormal and $A^*AB\subseteq BA^*A$, then $BA$ is
$k$-paranormal. \par \vskip 0.2 cm \noindent (2) \;If $AB$ is
$k-*$-paranormal and $A^*AB\subseteq BA^*A$, then $BA$ is
$k-*$-paranormal.
\end{proposition}
\begin{proof} (1)\;
Using the normality of $A=U|A|=|A|U$ ( $U$ unitary) and the fact
that $$A^*AB\subseteq BA^*A$$ we see that $BA=U^*ABU.$\par \vskip
0.2 cm \noindent Now we have
\begin{eqnarray*}
\|BAx\|^k&=&\|U^*ABUx\|^k\\&\leq& \|ABUx\|^k \\&\leq&
\|(AB)^kUx\|\\&\leq&\|(U^*ABU)^kx\|\\&=&\|(BA)^kx\|.
\end{eqnarray*}
(2)\;By similar  argument.
\end{proof}

\end{document}